\documentclass{amsart}
\usepackage{amssymb,amsmath,color}

\theoremstyle{plain}
\newtheorem{thm}{Theorem}[section]

\newtheorem{lem}[thm]{Lemma}

\newtheorem{prop}[thm]{Proposition}

\theoremstyle{definition}

\theoremstyle{remark}

\newtheorem{remark}{Remark}
\newtheorem*{remark*}{Remark}

\numberwithin{equation}{section}

        \newcommand{\field}[1]{{\mathbb{#1}}}
        \newcommand{\NN}{\field{N}}

        \newcommand{\RR}{\field{R}}

\allowdisplaybreaks

\begin{document}

\title[Eigenvalue estimates
for a three-dimensional magnetic operator]{Eigenvalue estimates for
a three-dimensional magnetic Schr\"odinger operator}

\author{Bernard Helffer}

\address{D\'epartement de Math\'ematiques, B\^atiment 425, Univ
Paris-Sud et CNRS, F-91405 Orsay C\'edex, France}
\email{Bernard.Helffer@math.u-psud.fr}

\author{Yuri A. Kordyukov}
\address{Institute of Mathematics, Russian Academy of Sciences, 112 Chernyshevsky
str. 450008 Ufa, Russia} \email{yurikor@matem.anrb.ru}

\thanks{Y.K. is partially supported by the Russian Foundation of Basic
Research (grants 09-01-00389 and 12-01-00519).}

\begin{abstract}
We consider a magnetic Schr\"odinger operator
$H^h=(-ih\nabla-\vec{A})^2$ with the Dirichlet boundary conditions
in an open set $\Omega \subset {\mathbb R}^3$, where $h>0$ is a
small parameter. We suppose that the minimal value $b_0$ of the
module $|\vec{B}|$ of the vector magnetic field $\vec{B}$ is
strictly positive, and there exists a unique minimum point of
$|\vec{B}|$, which is non-degenerate. The main result of the paper
is upper estimates for the low-lying eigenvalues of the operator
$H^h$ in the semiclassical limit. We also prove the existence of an
arbitrary large number of spectral gaps in the semiclassical limit
in the corresponding periodic setting.
\end{abstract}



\date{}
 \maketitle

\section{Preliminaries and main results}

\subsection{Main assumptions}
We would like to analyze the asymptotic behavior, in the
semiclassical regime, of the low-lying eigenvalues of the Dirichlet
realization of the magnetic Schr\"odinger operator in an open set
$\Omega$ in ${\mathbb R}^3$:
\[
H^h=\left(hD_{X_1}-A_1(X)\right)^2+\left(hD_{X_2}-A_2(X)\right)^2+\left(hD_{X_3}-A_3(X)\right)^2,
\]
where $\vec{A}=(A_1,A_2,A_3)\in C^\infty(\bar \Omega , {\mathbb
R}^3)$ is a magnetic potential and $h>0$ is a small parameter. We
will denote the coordinates in ${\mathbb R}^3$ as
$X=(X_1,X_2,X_3)=(x,y,z)\,$. Let $\vec{B}=\operatorname{rot} \vec
A=(B_1,B_2,B_3)$ be the corresponding vector magnetic field:
\[
B_{1}=\frac{\partial A_3}{\partial y}-\frac{\partial A_2}{\partial
z}\,,\quad B_{2}=\frac{\partial A_1}{\partial z}-\frac{\partial
A_3}{\partial x}\,,\quad B_{3}=\frac{\partial A_2}{\partial
x}-\frac{\partial A_1}{\partial y}\,.
\]
Put
\[
b_0=\min \{|\vec B(X)|\, :\, X\in \Omega\}.
\]
We assume that there exist a (connected) bounded domain
$\Omega_1\subset\subset \Omega$ and a constant $\epsilon_0>0$ such
that
\begin{equation}\label{ass1}
 |\vec B(X)| \geq b_0+\epsilon_0, \quad x\not\in \Omega_1\,.
\end{equation}
We also assume that:
\begin{equation}\label{ass2}
b_0>0\,,
\end{equation}
 and that there exists a unique minimum $X_0\in \Omega$ such that
  $|\vec B(X_0)|=b_0$, which is non-degenerate: in some neighborhood of
  $X_0$
\begin{equation}\label{ass3}
C^{-1}|X-X_0|^2\leq |\vec B(X)|-b_0 \leq C |X-X_0|^2\,.
\end{equation}

\subsection{Main statement}~\\
Denote
\[
d=\det {\rm Hess}\,|\vec B|(X_0)\,, \quad a=\frac{1}{2 b_0^2}({\rm
Hess} |\vec B|\, {\vec B}\cdot \vec B) (X_0)\,.
\]
Denote by $\lambda_0(H^h)\leq \lambda_1(H^h)\leq \lambda_2(H^h)\leq
\ldots$ the eigenvalues of the operator $H^h$.

\begin{thm}\label{t:main}
Under current assumptions, for any natural $m$, there exist $C_m>0$
and $h_m>0$ such that, for any $h\in (0,h_m]\,$,
\[
\lambda_m(H^h) \leq h b_0 +h^{3/2}
a^{1/2}+h^2\left[\frac{1}{2b_0}\left(\frac{d}{2a}\right)^{1/2}(2m+1)+
\nu \right]+C_m h^{9/4}\,,
\]
where $\nu $ is some explicit constant (which means that it is given
by a rather complicated explicit formula).
\end{thm}
Theorem \ref{t:main} will be based on a construction of quasimodes.
\begin{thm}\label{t:qmodes}
Under current assumptions, for any natural $j$, $k$ and $m$, there
exist $\phi^h_{j,k,m}\in C^\infty_c(\Omega)$, $C_{j,k,m}>0$ and
$h_{j,k,m}>0$ such that
\begin{equation}\label{e:orth}
(\phi^h_{j_1,k_1,m_1},\phi^h_{j_2,k_2,m_2})
=\delta_{j_1j_2}\delta_{k_1k_2}\delta_{m_1m_2}+ \mathcal
O_{j_1,k_1,m_1,j_2,k_2,m_2}(h)\,,
\end{equation}
and, for any $h\in (0,h_{j,k,m}]$,
\begin{equation}\label{e:Hh}
\|H^h\phi^h_{j,k,m}- \mu_{j,k,m}^h \phi^h_{j,k,m}\|\leq C_{j,k,m}\,
h^{\frac{9}{4}}\|\phi^h_{j,k,m}\|,\,
\end{equation}
where
\[
\mu_{j,k,m}^h=\mu_{j,k,m,0}h+\mu_{j,k,m,2} h^{\frac 3
2}+\mu_{j,k,m,4} h^2\,.
\]
with
\[
\mu_{j,k,m,0} = (2k+1)b_0\,,\quad
\mu_{j,k,m,2}=(2j+1)(2k+1)^{1/2}a^{1/2}\,,
\]
and
\[
\mu_{j,k,m,4}=\frac{1}{2b_0}\left(\frac{d}{2a}\right)^{1/2}(2m+1)(2k+1)+
\nu(j,k)\,,
\]
where $\nu(j,k)$ has the form
\[
\nu(j,k)=\nu_{22}(2k+1)^2+\nu_{11}(2j+1)^2+\nu_0
\]
with some explicit constants $\nu_0,\nu_{11},\nu_{22}$.
\end{thm}
\begin{remark}~\\
It is conjectured that
\[
\lambda_m(H^h) \geq hb_0
+h^{3/2}a^{1/2}+h^2\left[\frac{1}{2b_0}\left(\frac{d}{2a}\right)^{1/2}(2m+1)+
\nu \right]- C_m h^{9/4}\,.
\]
\end{remark}

\subsection{History of the problem}~\\
May be at the level of the mathematical problem, the starting
reference for the spectral analysis of self-adjoint realizations of
the magnetic Schr\"odinger operator is the paper by
Avron-Herbst-Simon \cite{AvHeSi} where the role of the module of the
magnetic field in the three-dimensional case appears for the first
time. Further investigations were inspired by R. Montgomery
\cite{Mont}, who was asking ``Can we hear the locus of the magnetic
field'' (by analogy with the celebrated question by M. Kac). In
\cite{Mont}, this question was studied for the two-dimensional
magnetic Schr\"odinger operator. Motivated by the question of R.
Montgomery, the first author and Mohamed in \cite{HM} investigated
the asymptotic behavior of the low-lying eigenvalues of the
Dirichlet realization of the magnetic Schr\"odinger operator in the
case when the magnetic field vanishes. This study was continued more
recently in \cite{PaKw,Qmath10,miniwells,DR} (see also
\cite{luminy}). The case when the magnetic field never vanishes was
analyzed in detail for the Dirichlet realization in the
two-dimensional case in \cite{HM01} and more recently in
\cite{HKI,HKII}. Moreover, there is a big literature devoted to the
spectral analysis of the Neumann realization because of its
connection with problems in superconductivity (see
\cite{Fournais-Helffer:book} and the references therein). Finally,
we do not also give a complete description of the semi-classical
results obtained in the case when an electric potential $V$ is
creating the main localization and refer to \cite{He} and
\cite{DiSj} for a presentation and references therein.

The paper is organized as follows. In Section~\ref{s:prelim1}, we
make some simplifications of the operator $H^h$, using linear
changes of variables and gauge transformations, and present the
normal form appearing in the generic situation.
Section~\ref{s:prelim2} is devoted to the analysis of the action of
the metaplectic group on our models and to the introduction of a
suitable rescaling of the problem.  In Section~\ref{s:eigenmodes} we
construct approximate eigenfunctions of the operator $H^h$,
completing the proof of Theorem~\ref{t:qmodes}. In
Section~\ref{s:gaps} we consider the case when the magnetic field is
periodic. We combine the constructions of approximate eigenfunctions
given in Sections~\ref{s:eigenmodes} with the results of
\cite{diff2006} to prove the existence of arbitrary large number of
gaps in the spectrum of the periodic operator $H^h$ in the
semiclassical limit.

\section{Preliminaries for the proof: Normal form}~
\label{s:prelim1} In this section, we make some simplifications of
the operator $H^h$, using linear changes of variables and gauge
transformations.

Without loss of generality we will assume that $X_0=(0,0,0)$ and the
magnetic field $\vec{B}$ at $X_0$ is $(0,0,b_0)\,$. Thus, we can
write the Taylor expansion up to order $4$ at $(0,0,0)$ of the
magnetic field:
\begin{equation}
\left\{
\begin{aligned}\label{e:formB}
B_1(X)& =\ell_1\cdot X+Q_1(X) +C_1(X)+R_1(X)+\mathcal O(|X|^5)\,,\\
B_2(X)& =\ell_2\cdot X +Q_2(X)+C_2(X)+R_2(X)+\mathcal O(|X|^5)\,,\\
B_3(X)& =b_0+Q_3(X)+C_3(X)+R_3(X)+\mathcal O(|X|^5)\,,
\end{aligned}
\right.
\end{equation}
where $\ell_j\in {\mathbb R}^3, j=1,2,$ and $Q_j$, $C_j$ and $R_j$,
$j=1,2,3,$ are the terms of order 2, 3 and 4, respectively.

By assumption \eqref{ass3}, we have
\begin{equation}\label{ass3a}
(\ell_1\cdot X)^2+(\ell_2\cdot X)^2+2b_0Q_3(X)
>0\,,\quad \forall X\in {\mathbb R}^3\,.
\end{equation}
We also have
\[
{\rm div}\,\vec B = \frac{\partial B_1}{\partial x}+\frac{\partial
B_2}{\partial y}+\frac{\partial B_3}{\partial z}=0\,,
\]
which implies
\begin{equation}\label{eq:1.8}\left\{
\begin{aligned}
\frac{\partial (\ell_1\cdot X)}{\partial x}+\frac{\partial
(\ell_2\cdot X)}{\partial y}& =0\,;\\
\frac{\partial Q_1}{\partial x}+\frac{\partial Q_2}{\partial
y}+\frac{\partial
Q_3}{\partial z}& =0\,;\\
\frac{\partial C_1}{\partial x}+\frac{\partial C_2}{\partial
y}+\frac{\partial C_3}{\partial z}& =0\,;\\
\frac{\partial R_1}{\partial x}+\frac{\partial R_2}{\partial
y}+\frac{\partial R_3}{\partial z}& =0 \,.
\end{aligned}
\right.
\end{equation}
With the notation,
\[
\ell_1=(\alpha_1,\alpha_2,\alpha_3)\,, \quad
\ell_2=(\beta_1,\beta_2,\beta_3)\,,
\]
the first line of \eqref{eq:1.8} gives
\begin{equation}
\alpha_1+\beta_2=0\,.
\end{equation}
We can assume without lost of generality that
\begin{equation}
\alpha_3=0\,.
\end{equation}
Indeed, either it is already the case, either $\beta_3=0$ and we
can exchange the role of $x$ and $y$, or $\alpha_3\neq 0$ and
$\beta_3\neq 0$ and a rotation around the $z$-axis permits to obtain
the cancellation of $\alpha_3$.

With the notation
\begin{equation}
Q_1(X)=\sum_{i,j=1}^3a_{ij}X_iX_j\,, \;
Q_2(X)=\sum_{i,j=1}^3b_{ij}X_iX_j\,, \;
Q_3(X)=\sum_{i,j=1}^3c_{ij}X_iX_j\,,
\end{equation}
the second line of \eqref{eq:1.8} reads
\begin{equation}
a_{1j}+b_{2j}+c_{3j}=0\,, \quad j=1,2,3\,.
\end{equation}
With this notation, \eqref{ass3} is equivalent to saying that the
matrix
\begin{equation}\label{eq:Q}
Q=\frac12 {\rm Hess} (|\vec B|^2)(0)=b_0\,{\rm Hess} |\vec B|(0)
\end{equation}
or equivalently (cf. \eqref{ass3a})
\[
Q=
\begin{pmatrix}
\alpha_1^2+\beta_1^2+2b_0c_{11} & \alpha_1\alpha_2+\beta_1\beta_2+2b_0c_{12} & \beta_1\beta_3+2b_0c_{13}\\
\alpha_1\alpha_2+\beta_1\beta_2+2b_0c_{21} & \alpha_2^2+\beta_2^2+2b_0c_{22} & \beta_2\beta_3+2b_0c_{23}\\
\beta_1\beta_3+2b_0c_{31} & \beta_2\beta_3+2b_0c_{32} & \beta_3^2+2b_0c_{33}
\end{pmatrix}
\]
is positive definite. We will denote the elements of $Q$ by
$\{q_{ij}\}_{i,j=1,2,3}$.

Without loss of generality, we can, after possibly a gauge transformation, assume that:
\begin{equation*}
A_1=0\,.
\end{equation*}
With this assumption, we get an equation for $A_2$ by writing:
\[
B_{3}=\frac{\partial A_2}{\partial x}=b_0+
Q_3(X)+C_3(X)+R_3(X)+\mathcal O(|X|^5)\,,
\]
and without loss of generality, we can, after possibly a gauge
transformation assume that
\begin{equation*}
A_2(0,y,z) =0\,.
\end{equation*}
This implies
\begin{multline}\label{eq:1.14}
A_2(x,y,z) \\ =b_0x+\int_0^x  Q_3(\hat X)\,d\hat x +\int_0^x
C_3(\hat X)\,d\hat x+\int_0^x R_3(\hat X)\,d\hat x +x\, \mathcal
O(|X|^5)\,,
\end{multline}
where $\hat X =( \hat x, y, z)$\,. Similarly, we get an equation for $A_3$ by writing:
\[
B_{2}(x,y,z) =-\frac{\partial A_3}{\partial x}=\ell_2\cdot X +
Q_2(X)+C_2(X)+R_2(X)+\mathcal O(|X|^5)
\]
and without loss of generality we can, after possibly a gauge
transformation (keeping the previous properties)  assume that
$$
A_3(0,0,z) =0\,.
$$
This implies
\begin{equation}\label{eq:1.16}
\begin{split}
 A_3(x,y,z) = & -\int_0^x\ell_2\cdot \hat X \,d\hat x -\int _0^x
Q_2(\hat X)\, d\hat x-\int _0^x C_2(\hat X)\,d\hat x \\ & -\int _0^x
R_2(\hat X)\,d\hat x +a_3(y,z)+x\, \mathcal O(|X|^5)\,,
\end{split}
\end{equation}
where $\hat X =( \hat x, y, z)$ and $a_3(y,z)=A_3(0,y,z)$ satisfies
\begin{equation}\label{eq:1.17}
a_3(0,z)=0\,.
\end{equation}
From \eqref{eq:1.14} and \eqref{eq:1.16}, we obtain, using the
definition of $B_1$,
\begin{align*}
B_{1}(x,y,z) =& \frac{\partial A_3}{\partial y}(X) -\frac{\partial A_2}{\partial
z}(X) \\
=& -\beta_2x -\int_0^x  \frac{\partial Q_2}{\partial y}
(\hat X) d\hat x-\int_0^x  \frac{\partial C_2}{\partial y} (\hat X)\,d\hat x
-\int_0^x  \frac{\partial R_2}{\partial y} (\hat X)\,d\hat x\\
& +\frac{\partial a_3}{\partial y}(y,z) -\int _0^x \frac{\partial
Q_3}{\partial z} (\hat X)\,d\hat x-\int_0^x \frac{\partial
C_3}{\partial z} (\hat X)\,d\hat x\\ & -\int_0^x \frac{\partial
R_3}{\partial z} (\hat X)\,d\hat x +x\, \mathcal O(|X|^4)\,,
\end{align*}
and using \eqref{eq:1.8}, we obtain
\begin{align*}
B_1(x,y,z) = & \alpha_1x +Q_1(X)-
Q_1(0,y,z) +C_1(X)- C_1(0,y,z)\\
& +R_1(X)- R_1(0,y,z)+\frac{\partial a_3}{\partial y}(y,z) +\mathcal
O(|X|^5)\,.
\end{align*}
On the other hand, comparing with \eqref{e:formB}, we obtain
\[
\frac{\partial a_3}{\partial y}(y,z) =\alpha_2y+Q_1(0,y,z)+ C_1(0,y,z)+ R_1(0,y,z)+\mathcal
O(|y|^5+|z|^5)\,.
\]
Thus, we obtain, having \eqref{eq:1.17} in mind,
\begin{multline}\label{eq:1.20}
 a_3(y,z)=\frac12\alpha_2y^2 + \int_0^y
Q_1(0,\hat y,z)\, d\hat y\\ + \int_0^y C_1(0,\hat y,z)\, d\hat y+
\int_0^y R_1(0,\hat y,z)\, d\hat y+y\, \mathcal O(|y|^5+|z|^5)\,.
 \end{multline}

Coming back to \eqref{eq:1.14}, we have
\begin{align*}
A_2(x,y,z) =& b_0x+
\frac13c_{11}x^3+c_{12}x^2y+c_{13}x^2z+c_{22}xy^2+2c_{23}xyz+c_{33}xz^2\\
& +\int_0^x C_3(\hat X)\,d\hat x +\int_0^x R_3(\hat X)\,d\hat x
+x\,\mathcal O(|X|^5)\,.
\end{align*}
Similarly, coming back to \eqref{eq:1.16} and using \eqref{eq:1.20}, we obtain
\begin{align*}
    A_3(x,y,z)
=& -\frac12\beta_1x^2-\beta_2xy-\beta_3xz +\frac12\alpha_2y^2\\
& +\int_0^y C_1(0,\hat y,z) d\hat y +\int_0^y R_1(0,\hat y,z) d\hat
y
\\ &
-\frac13b_{11}x^3-b_{12}x^2y-b_{13}x^2z-b_{22}xy^2-2b_{23}xyz-b_{33}xz^2\\
& + \frac13a_{22}y^3+a_{23}y^2z+a_{33}yz^2  -\int_0^x C_2(\hat
X)\,d\hat x-\int_0^x  R_2(\hat X)\,d\hat x\\ & +x\,\mathcal
O(|X|^5)+y\, \mathcal O(|y|^5+|z|^5)\,.
\end{align*}
Finally, what we have got is the following normal form:
\begin{prop}
Using only linear change of variables and gauge transformations, we can assume that
\[
A_1(x,y,z)=0\,, \quad A_2(0,y,z) =0\,, \quad A_3(0,0,z) =0\,.
\]
Moreover:
\begin{equation}\label{e:normal_form}
\begin{aligned}
A_2(x,y,z) = & b_0x+
\frac13c_{11}x^3+c_{12}x^2y+c_{13}x^2z+c_{22}xy^2+2c_{23}xyz+c_{33}xz^2\\
& +\int_0^x C_3(\hat X)\,d\hat x +\int_0^x R_3(\hat X)\,d\hat x
+x\, \mathcal O(|X|^5)\,,\\
A_3(x,y,z) = & -\frac12\beta_1x^2-\beta_2xy-\beta_3xz\\
& +\frac12\alpha_2y^2+\int_0^y C_1(0,\hat y,z) d\hat y +\int_0^y
R_1(0,\hat y,z) d\hat y \\ &
-\frac13b_{11}x^3-b_{12}x^2y-b_{13}x^2z-b_{22}xy^2-2b_{23}xyz-b_{33}xz^2\\
& + \frac13a_{22}y^3+a_{23}y^2z+a_{33}yz^2-\int_0^x C_2(\hat
X)\,d\hat x -\int_0^x  R_2(\hat X)\,d\hat x \\ & +x\,\mathcal
O(|X|^5)+y\,\mathcal O(|y|^5+|z|^5)\,.
\end{aligned}
\end{equation}
\end{prop}

\section{Preliminaries for the proof: Metaplectic transformations} \label{s:prelim2}
In this section, we make further simplifications of the operator
$H^h$, using metaplectic transformations. First, we introduce a
suitable $h$-dependent rescaling of the problem and expand the
resulting operator in fractional powers of $h$.

\subsection{Rescaling}~\\ Let us make a change of variables
\[
x=h^{\frac12}\tilde{x},\quad y=h^{\frac12}\tilde{y},\quad
z=h^{\frac14}\tilde{z}.
\]

We will only apply our operator $H^h$ on functions which are a
product of cut-off functions localized in a neighborhood of
$(0,0,0)\in {\mathbb R}^3$ with linear combinations of terms like
\[
h^\nu w(h^{-1/2}x,h^{-1/2}y,h^{-1/4}z)\,,
\]
with $w$ in $\mathcal S ({\mathbb R}^3)$. These functions are
consequently $\mathcal O(h^\infty)$ outside a fixed neighborhood of
$(0,0,0)$. We will start by doing the computations formally in the
sense that everything is determined modulo $\mathcal O(h^\infty)$,
and any smooth function will be replaced by its Taylor's expansion
at $(0,0,0)\,$. We introduce
\begin{gather*}
C_1(X)=\sum_{i\leq j\leq k}p_{ijk}X_iX_jX_k\, ,\\
C_2(X)=\sum_{i\leq j\leq k}q_{ijk}X_iX_jX_k\, , \\
C_3(X)=\sum_{i\leq j\leq k}r_{ijk}X_iX_jX_k\, .
\end{gather*}
We will also need the coefficient of $z^4$ in $R_j$:
\[
R_j(X)=\delta_jz^4+\ldots, \quad j=1,2,3\,.
\]
We have
\[
H^h=hP^{h}\,,
\]
where
\begin{multline*}
P^{h}= D_{\tilde {x}}^2 +\left(D_{\tilde{y}}
-h^{-1/2}A_2(h^{\frac12}\tilde{x},h^{\frac12}\tilde{y},h^{\frac14}\tilde{z})\right)^2\\
 +\left(h^{1/4}D_{\tilde{z}}-h^{-1/2}
A_3(h^{\frac12}\tilde{x},h^{\frac12}\tilde{y},h^{\frac14}\tilde{z})\right)^2.
\end{multline*}

Now, using \eqref{e:normal_form}, we expand the operator in
fractional powers of $h$. First, we compute
\begin{multline*}
h^{-1/2}A_2(h^{\frac12}\tilde{x},h^{\frac12}\tilde{y},h^{\frac14}\tilde{z})\\
\begin{aligned}
\quad \quad\quad  =& \, b_0\tilde{x}+ h^{\frac12}
c_{33}\tilde{x}\tilde{z}^2+h^{\frac34}[c_{13}\tilde{x}^2\tilde{z}+2c_{23}\tilde{x}\tilde{y}\tilde{z}
+r_{333}\tilde{x}\tilde{z}^3]\\ &
+h\left[\frac13c_{11}\tilde{x}^3+c_{12}\tilde{x}^2\tilde{y}
 +c_{22}\tilde{x}\tilde{y}^2+\frac12r_{133}\tilde{x}^2\tilde{z}^2
 +r_{233}\tilde{x}\tilde{y}\tilde{z}^2+\delta_3\tilde{x}\tilde{z}^4\right]\\
& +\mathcal O(h^{5/4})\,,
\end{aligned}
\end{multline*}
and
\begin{multline*}
h^{-1/2}A_3(h^{\frac12}\tilde{x},h^{\frac12}\tilde{y},h^{\frac14}\tilde{z})\\
\begin{aligned}
=& -h^{\frac14}\beta_3\tilde{x}\tilde{z}+h^{\frac12}\left[
-\frac12\beta_1\tilde{x}^2-\beta_2\tilde{x}\tilde{y}
+\frac12\alpha_2\tilde{y}^2-b_{33}\tilde{x}\tilde{z}^2+a_{33}\tilde{y}\tilde{z}^2\right]\\
&
+h^{\frac34}\left[-b_{13}\tilde{x}^2\tilde{z}-2b_{23}\tilde{x}\tilde{y}\tilde{z}+a_{23}\tilde{y}^2\tilde{z}
-q_{333}\tilde{x}\tilde{z}^3+p_{333}\tilde{y}\tilde{z}^3\right]
+\mathcal O(h)\,.
\end{aligned}
\end{multline*}
Using these formulae and omitting the tilda's, we obtain (in the
dilated coordinates)
\[
P^h=P_0+h^{\frac14}P_1+h^{\frac12}P_2+h^{\frac34}P_3+hP_4+\mathcal
O(h^{5/4})\,,
\]
where
\begin{align*}
P_0=& D_{x}^2+(D_{{y}}-b_0{x})^2\,, \\
P_1=& 0\,,\\
P_2=& -2c_{33}{x}{z}^2(D_{{y}}-b_0{x})+(D_{{z}}+\beta_3{x}{z})^2\,,\\
P_3=& -2c_{13}{x}^2{z}(D_{{y
}}-b_0{x})-2c_{23}[(D_{{y}}-b_0{x}){x}{y}{z}+xyz(D_{{y}}-b_0{x})]
\\ & -2r_{333}{x}{z}^3(D_{{y}}-b_0{x})+(\beta_1{x}^2+2\beta_2{x}{y}
-\alpha_2{y}^2)(D_{{z}}+\beta_3{x}{z})\\
& +(b_{33}{x}-a_{33}y)((D_{{z}}+\beta_3{x}{z}){z}^2+z^2(D_{{z}}+\beta_3{x}{z}))\, ,\\
P_4=& -(D_{{y}}-b_0{x})\left[\frac13c_{11}{x}^3+c_{12}{x}^2{y}
+c_{22}{x}{y}^2+\frac12r_{133}{x}^2{z}^2
+r_{233}{x}{y}{z}^2+\delta_3{x}{z}^4\right]\\ & -
\left[\frac13c_{11}{x}^3+c_{12}{x}^2{y}
+c_{22}{x}{y}^2+\frac12r_{133}{x}^2{z}^2
+r_{233}{x}{y}{z}^2+\delta_3{x}{z}^4\right](D_{{y}}-b_0{x})\\
&+c_{33}^2{x}^2{z}^4
 +
(D_{{z}}+\beta_3{x}{z})\left[b_{13}{x}^2{z}+2b_{23}{x}{y}{z}-a_{23}y^2z
+q_{333}{x}{z}^3-p_{333}yz^3\right]\\ &
+\left[b_{13}{x}^2{z}+2b_{23}{x}{y}{z}-a_{23}y^2z
+q_{333}{x}{z}^3-p_{333}yz^3\right](D_{{z}}+\beta_3{x}{z})\\
& + \left( \frac12\beta_1{x}^2+\beta_2{x}{y}
-\frac12\alpha_2{y}^2+b_{33}{x}{z}^2-a_{33}{y}{z}^2\right)^2.
\end{align*}
\subsection{Partial Fourier transform and gauge transform}~\\
Now we are going to apply some metaplectic transformations. We
recall that these transformations are unitary and therefore preserve
the spectrum of the operators.

First, we make a partial Fourier transform in $y$:
\begin{align*}
\widehat P_0=& D_{x}^2+(\eta-b_0{x})^2\,, \\
\widehat P_1=& 0\,,\\
\widehat P_2=& -2c_{33}{x}{z}^2(\eta-b_0{x})+(D_{{z}}+\beta_3{x}{z})^2\,,\\
\widehat P_3=&
-2c_{13}{x}^2{z}(\eta-b_0{x})-2c_{23}[(\eta-b_0{x}){x}(-D_\eta){z}+x(-D_\eta)z(\eta-b_0{x})]
\\ & -2r_{333}{x}{z}^3(\eta-b_0{x})+(\beta_1{x}^2+2\beta_2{x}(-D_\eta)
-\alpha_2D_\eta^2)(D_{{z}}+\beta_3{x}{z})\\
& +(b_{33}{x}+a_{33}D_\eta)((D_{{z}}+\beta_3{x}{z}){z}^2+z^2(D_{{z}}+\beta_3{x}{z}))\, ,\\
\widehat P_4=&
-(\eta-b_0{x})\Big[\frac13c_{11}{x}^3+c_{12}{x}^2(-D_\eta)
+c_{22}{x}D_\eta^2+\frac12r_{133}{x}^2{z}^2\\
& +r_{233}{x}(-D_\eta){z}^2+\delta_3{x}{z}^4\Big] -
\Big[\frac13c_{11}{x}^3+c_{12}{x}^2(-D_\eta)
+c_{22}{x}D_\eta^2+\frac12r_{133}{x}^2{z}^2\\
& +r_{233}{x}(-D_\eta){z}^2+\delta_3{x}{z}^4\Big](\eta-b_0{x})+c_{33}^2{x}^2{z}^4\\
& +
(D_{{z}}+\beta_3{x}{z})\left[b_{13}{x}^2{z}+2b_{23}{x}(-D_\eta){z}-a_{23}D^2_\eta
z +q_{333}{x}{z}^3-p_{333}(-D_\eta)z^3\right]\\ &
+\left[b_{13}{x}^2{z}+2b_{23}{x}(-D_\eta){z}-a_{23}D^2_\eta z
+q_{333}{x}{z}^3-p_{333}(-D_\eta)z^3\right](D_{{z}}+\beta_3{x}{z})\\
& + \left( \frac12\beta_1{x}^2+\beta_2{x}(-D_\eta)
-\frac12\alpha_2D_\eta^2+b_{33}{x}{z}^2+a_{33}D_\eta z^2\right)^2\,.
\end{align*}
Next, we make a translation $\check{x}=x-\frac{\eta}{b_0}$ (and forget the "check"):
\begin{align*}
\check P_0=& D_{x}^2+b^2_0{x}^2\,, \\
\check P_1=& 0\,,\\
\check P_2=& 2c_{33}(x+\frac{\eta}{b_0}){z}^2b_0x+(D_{{z}}+\beta_3(x+\frac{\eta}{b_0}){z})^2\,,\\
\check P_3=&
2c_{13}(x+\frac{\eta}{b_0})^2{z}(b_0{x})+2c_{23}\Big[(b_0{x})(x+\frac{\eta}{b_0})
\Big(\frac{1}{b_0}D_x-D_\eta\Big){z}\\
&
+(x+\frac{\eta}{b_0})\Big(\frac{1}{b_0}D_x-D_\eta\Big)z(b_0{x})\Big]
+2r_{333}(x+\frac{\eta}{b_0}){z}^3(b_0{x})\\
&  +\Big(\beta_1(x+\frac{\eta}{b_0})^2
+2\beta_2(x+\frac{\eta}{b_0})\left(\frac{1}{b_0}D_x-D_\eta\right)
-\alpha_2\left(\frac{1}{b_0}D_x-D_\eta\right)^2\Big)\times \\
& \times (D_{{z}}+\beta_3(x+\frac{\eta}{b_0}){z}) +
\Big(b_{33}(x+\frac{\eta}{b_0}) -
a_{33}\Big(\frac{1}{b_0}D_x-D_\eta\Big)\Big)\times
\\ & \times \Big((D_{{z}}+\beta_3(x+\frac{\eta}{b_0}){z}){z}^2
+z^2(D_{{z}}+\beta_3(x+\frac{\eta}{b_0}){z})\Big)\,,\\
\check P_4=& b_0{x}\Big[\frac13c_{11}(x+\frac{\eta}{b_0})^3
+c_{12}(x+\frac{\eta}{b_0})^2\left(\frac{1}{b_0}D_x-D_\eta\right)\\
& +c_{22}(x+\frac{\eta}{b_0})\left(\frac{1}{b_0}D_x-D_\eta\right)^2
 +\frac12r_{133}(x+\frac{\eta}{b_0})^2{z}^2\\ &
+r_{233}(x+\frac{\eta}{b_0})\left(\frac{1}{b_0}D_x-D_\eta\right){z}^2
+\delta_3(x+\frac{\eta}{b_0}){z}^4\Big]\\
& + \Big[\frac13c_{11}(x+\frac{\eta}{b_0})^3
+c_{12}(x+\frac{\eta}{b_0})^2\left(\frac{1}{b_0}D_x-D_\eta\right)\\
& +c_{22}(x+\frac{\eta}{b_0})\left(\frac{1}{b_0}D_x-D_\eta\right)^2
 +\frac12r_{133}(x+\frac{\eta}{b_0})^2{z}^2\\
&
+r_{233}(x+\frac{\eta}{b_0})\left(\frac{1}{b_0}D_x-D_\eta\right){z}^2
+\delta_3(x+\frac{\eta}{b_0}){z}^4\Big] (b_0{x})\\ &
+c_{33}^2(x+\frac{\eta}{b_0})^2{z}^4 +
(D_{{z}}+\beta_3(x+\frac{\eta}{b_0}){z})\Big[b_{13}(x+\frac{\eta}{b_0})^2{z}
\\ & +2b_{23}(x+\frac{\eta}{b_0})\left(\frac{1}{b_0}D_x-D_\eta\right){z}
 -a_{23}\left(\frac{1}{b_0}D_x-D_\eta\right)^2 z\\ & +q_{333}(x+\frac{\eta}{b_0}){z}^3-p_{333}\left(\frac{1}{b_0}D_x-D_\eta\right)z^3\Big]\\
&
+\Big[b_{13}(x+\frac{\eta}{b_0})^2{z}+2b_{23}(x+\frac{\eta}{b_0})\left(\frac{1}{b_0}D_x-D_\eta\right){z}
-a_{23}\left(\frac{1}{b_0}D_x-D_\eta\right) ^2 z \\
&
+q_{333}(x+\frac{\eta}{b_0}){z}^3-p_{333}\left(\frac{1}{b_0}D_x-D_\eta\right)z^3\Big](D_{{z}}+\beta_3(x+\frac{\eta}{b_0}){z})\\
& + \Big(
\frac12\beta_1(x+\frac{\eta}{b_0})^2+\beta_2(x+\frac{\eta}{b_0})\left(\frac{1}{b_0}D_x-D_\eta\right)
-\frac12\alpha_2\left(\frac{1}{b_0}D_x-D_\eta\right)^2\\
& +b_{33}(x+\frac{\eta}{b_0}){z}^2
-a_{33}\left(\frac{1}{b_0}D_x-D_\eta\right) z^2\Big)^2.
\end{align*}
Finally, we make a gauge transformation (by
$\exp(-i\frac{\beta_3}{2b_0}\eta z^2)$):
\begin{align*}
\widetilde P_0= & D_{x}^2+b^2_0{x}^2\,,\\
\widetilde P_1= & 0\,,\\
\widetilde P_2= &
2b_0c_{33}x(x+\frac{\eta}{b_0}){z}^2+(D_{{z}}+\beta_3x{z})^2\,,\\
\widetilde P_3 = &
2b_0c_{13}x(x+\frac{\eta}{b_0})^2{z}+2c_{23}\Big[(b_0{x})(x+\frac{\eta}{b_0})
\left(\frac{1}{b_0}D_x+\frac{\beta_3}{2b_0} z^2-D_\eta\right){z}\\
& +(x+\frac{\eta}{b_0})\left(\frac{1}{b_0}D_x+\frac{\beta_3}{2b_0}
z^2-D_\eta\right)z(b_0{x})\Big]
+2r_{333}(x+\frac{\eta}{b_0}){z}^3(b_0{x})\\ &
+\Big(\beta_1(x+\frac{\eta}{b_0})^2
+2\beta_2(x+\frac{\eta}{b_0})\left(\frac{1}{b_0}D_x+\frac{\beta_3}{2b_0}
z^2-D_\eta\right)\\
& -\alpha_2\left(\frac{1}{b_0}D_x+\frac{\beta_3}{2b_0}
z^2-D_\eta\right)^2\Big)(D_{{z}}+\beta_3x{z})\\
& + \left(b_{33}(x+\frac{\eta}{b_0}) -
a_{33}\left(\frac{1}{b_0}D_x+\frac{\beta_3}{2b_0}
z^2-D_\eta\right)\right)\times \\ & \times
((D_{{z}}+\beta_3x{z}){z}^2
+z^2(D_{{z}}+\beta_3x{z})),\\
 \widetilde P_4 = & b_0{x}\Big[\frac13c_{11}(x+\frac{\eta}{b_0})^3
+c_{12}(x+\frac{\eta}{b_0})^2\left(\frac{1}{b_0}D_x+\frac{\beta_3}{2b_0}
z^2-D_\eta\right)\\
&
+c_{22}(x+\frac{\eta}{b_0})\left(\frac{1}{b_0}D_x+\frac{\beta_3}{2b_0}
z^2-D_\eta\right)^2 +\frac12r_{133}(x+\frac{\eta}{b_0})^2{z}^2\\
&
+r_{233}(x+\frac{\eta}{b_0})\left(\frac{1}{b_0}D_x+\frac{\beta_3}{2b_0}
z^2-D_\eta\right){z}^2
+\delta_3(x+\frac{\eta}{b_0}){z}^4\Big]\\
& + \Big[\frac13c_{11}(x+\frac{\eta}{b_0})^3
+c_{12}(x+\frac{\eta}{b_0})^2\left(\frac{1}{b_0}D_x+\frac{\beta_3}{2b_0}
z^2-D_\eta\right)\\
&
+c_{22}(x+\frac{\eta}{b_0})\left(\frac{1}{b_0}D_x+\frac{\beta_3}{2b_0}
z^2-D_\eta\right)^2 +\frac12r_{133}(x+\frac{\eta}{b_0})^2{z}^2\\
&
+r_{233}(x+\frac{\eta}{b_0})\left(\frac{1}{b_0}D_x+\frac{\beta_3}{2b_0}
z^2-D_\eta\right){z}^2 +\delta_3(x+\frac{\eta}{b_0}){z}^4\Big]
(b_0{x})\\ & +c_{33}^2(x+\frac{\eta}{b_0})^2{z}^4
 + (D_{{z}}+\beta_3x{z})\Big[b_{13}(x+\frac{\eta}{b_0})^2{z}\\ &
+2b_{23}(x+\frac{\eta}{b_0})\left(\frac{1}{b_0}D_x+\frac{\beta_3}{2b_0}
z^2-D_\eta\right){z}
-a_{23}\left(\frac{1}{b_0}D_x+\frac{\beta_3}{2b_0}z^2-D_\eta\right)
^2 z \\ & +q_{333}(x+\frac{\eta}{b_0}){z}^3
-p_{333}\left(\frac{1}{b_0}D_x+\frac{\beta_3}{2b_0}z^2-D_\eta\right)z^3\Big]
+\Big[b_{13}(x+\frac{\eta}{b_0})^2{z}\\
& +2b_{23}(x+\frac{\eta}{b_0})
\left(\frac{1}{b_0}D_x+\frac{\beta_3}{2b_0}z^2-D_\eta\right){z}
-a_{23}\left(\frac{1}{b_0}D_x+\frac{\beta_3}{2b_0}z^2-D_\eta\right) ^2 z\\
&+q_{333}(x+\frac{\eta}{b_0}){z}^3-p_{333}\left(\frac{1}{b_0}D_x
+\frac{\beta_3}{2b_0}z^2-D_\eta\right)z^3\Big](D_{{z}}+\beta_3x{z})\\
& + \Big(
\frac12\beta_1(x+\frac{\eta}{b_0})^2+\beta_2(x+\frac{\eta}{b_0})\left(\frac{1}{b_0}D_x+\frac{\beta_3}{2b_0}
z^2-D_\eta\right)\\ &
-\frac12\alpha_2\left(\frac{1}{b_0}D_x+\frac{\beta_3}{2b_0}
z^2-D_\eta\right)^2 +b_{33}(x+\frac{\eta}{b_0}){z}^2\\
& -a_{33}\left(\frac{1}{b_0}D_x+\frac{\beta_3}{2b_0}
z^2-D_\eta\right) z^2\Big)^2\,.
\end{align*}

\section{Construction of approximate eigenfunctions} \label{s:eigenmodes}
In this section, we complete the proof of Theorem~\ref{t:qmodes}. We
start with a formal construction. We will construct an
approximate eigenfunction of the operator $P^h$ in the form
$$
u^h(x,\eta,z)\sim \sum_{ j\in \mathbb N} h^{\frac j4}
u_j(x,\eta,z)\,,
$$
corresponding to an approximate eigenvalue
$$
\lambda^h \sim \sum_{ j\in \mathbb N} h^{\frac j4} \lambda_j\,.
$$
We will express the cancelation of the coefficients of
$h^{{\ell}/{4}}$ ($\ell \in \mathbb N$) in the formal expansion of
$(P^{h}-\lambda^h)u^h$, starting with $\ell= 0$\,.

\subsection{The first equations}~\\
The first equations read (here we omit tilda's):
\begin{align}
\label{eq1} (P_0-\lambda_0)u_0 = & 0\,,\\
\label{eq2} (P_0-\lambda_0) u_1 = & \lambda_1 u_0\,,\\
\label{eq3} (P_0 -\lambda_0) u_2 = & - P_2 u_0  +\lambda_1 u_1 +\lambda_2 u_0\,,\\
\label{eq4} (P_0 -\lambda_0) u_3 = & - P_3 u_0 - P_2 u_1 +\lambda_1 u_2 +\lambda_2 u_1 +\lambda_3 u_0\,,\\
\label{eq5} (P_0 -\lambda_0) u_4 = &- P_4 u_0 - P_3 u_1 - P_2
u_2+\lambda_1 u_3  +\lambda_2 u_2 +\lambda_3 u_1 + \lambda_4 u_0\,,
\end{align}
and we now show how to solve them successively.

\subsection{Main  equation}~\\
The first equation \eqref{eq1} reads
\[
(D_x^2 +b_0^2 x^2-\lambda_0)u_0 = 0\,.
\]
We arrive at the eigenvalue problem for the harmonic oscillator
\begin{equation}\label{tnf13}
\mathfrak h_1:=D_x^2 +b_0^2x^2\,.
\end{equation}
Recall that its eigenvalues have the form
\[
{\rm Sp}(\mathfrak h_1)=\{\mu_{k}=(2k+1)b_0: k\in \NN\}\,.
\]
An eigenfunction of $\mathfrak h_1$ associated with  the eigenvalue
$\mu_{k}$ is given by
\[
\phi_{k}(x)=b_0^{1/2} H_{k}(b_0^{1/2}x) e^{-b_0x^2/2}\,,
\]
where $H_k$ is the Hermite polynomial:
\[
H_k(t)=(-1)^ke^{t^2}\frac{d^k}{dt^k}(e^{-t^2})\,.
\]
The norm of $\phi_{k}$ in $L^2(\RR,dx)$ equals the norm of $H_k$ in
$L^2(\RR,e^{-t^2}\,dt)$\,, which is given by
\[
\|H_k\|=\sqrt{2^kk!\sqrt{\pi}}\,.
\]
Some well-known formulae, concerning to the Hermite functions, which
we need in the paper, are gathered in the appendix.

We will take $\lambda_0$ as the $k$-th eigenvalue of $\mathfrak h_1$
\begin{equation}\label{tnf12}
\lambda_0=(2k+1)b_0\,
\end{equation}
and will look for a solution $u_0$ to \eqref{eq1} in the form
\begin{equation}\label{tnf11}
u_0(x,\eta,z) = \phi_k(x)\chi_0(\eta)\Psi_0(z)\,,
\end{equation}
with $\chi_0$ and $\Psi_0$ in $ \mathcal S (\mathbb R)$ to be
determined in the next steps.

Remark that what we do at this step corresponds to the so-called
reduction to the $k$-th Landau-Level.

\subsection{Second equation : coefficient of $h^{\frac 14}$}~\\
We look for a solution $u_1$ to \eqref{eq2} in the form
\begin{equation}\label{tnf14}
u_1(x,\eta,z) = \phi_k(x)v_1(\eta,z)\,,
\end{equation}
with $v_1\in \mathcal S(\mathbb R^2)$ to be determined in the next
steps, and take
\begin{equation}\label{tnf15}
\lambda_1=0\,.
\end{equation}

\subsection{Third  equation: coefficient of $h^{\frac 12}$}~\\
In order to solve the equation \eqref{eq3}, we find as a necessary
condition that, for any $z$ we should have,
\begin{equation}\label{eq3a}
-\frac{1}{\|H_k\|^2}\langle P_2u_0, \phi_k\rangle_x+ \lambda_2
\chi_0(\eta)\Psi_0(z) = 0\,,
\end{equation}
with $\langle\cdot,\cdot\rangle_x$ denoting the scalar product in
$L^2(\mathbb R_x)$. Using that $\phi_k^2$ is even, we obtain that
\begin{equation}
\begin{split}\label{eq3b}
 \frac{1}{\|H_k\|^2}\langle P_2u_0, \phi_k\rangle_x = & \frac{1}{\|H_k\|^2} \langle [
2b_0c_{33}x^2{z}^2+D^2_{{z}}+\beta_3x^2{z}^2] u_0, \phi_k\rangle_x\\
= & \chi_0(\eta)\mathfrak h_3\Psi_0(z) \,,
\end{split}
\end{equation}
where
\begin{equation}\label{tnf17}
\mathfrak h_3:=  D_z^2 + \frac{2k+1}{2b_0}(\beta_3^2
+2c_{33}b_0)z^2\,.
\end{equation}
 Thus, the condition \eqref{eq3a} can be rewritten as
\begin{equation}\label{tnf16}
\mathfrak h_3\Psi_0 =\lambda_2 \Psi_0\,.
\end{equation}
This determines $\lambda_2$ and $\Psi_0$. We take $\lambda_2$ as the
$j$-th eigenvalue eigenvalue of $\mathfrak h_3$:
\begin{equation}\label{tnf18}
\lambda_2 =(2j+1)(2k+1)^{1/2}\sqrt{\frac{\beta_3^2
+2c_{33}b_0}{2b_0}}=(2j+1)\Lambda_2\;,
\end{equation}
where
\[
\Lambda_2:=(2k+1)^{1/2}\sqrt{\frac{\beta_3^2
+2c_{33}b_0}{2b_0}}=(2k+1)^{1/2}\sqrt{\frac{q_{33}}{2b_0}}\;,
\]
and $\Psi_0$ as the corresponding eigenfunction:
\[
\Psi_0(z)=\psi_{j}(z):=\Lambda_2^{1/2} H_{j}(\Lambda_2^{1/2}z)
e^{-\Lambda_2z^2/2}\,.
\]
Observe that
\[
q_{33}=\frac12\frac{\partial^2|B|^2}{\partial z^2}(X_0)=2b_0a\,.
\]
Thus we obtain
\[
\lambda_2=(2j+1)(2k+1)^{1/2}a^{1/2}\,.
\]
To find $u_2$, from \eqref{eq3} we obtain that
\begin{multline*}
(\mathfrak h_1 -\lambda_0) u_2 \\
\begin{aligned}
= &
-\phi_k(x)\chi_0(\eta)D_{z}^2\psi_j(z)-x\phi_k(x)\chi_0(\eta)(2c_{33}\eta
{z}^2+\beta_3(D_zz+zD_z))\psi_j(z)\\ &
-(\beta_3^2+2b_0c_{33})x^2\phi_k(x)\chi_0(\eta)z^2\psi_j(z)+\lambda_2
\phi_k(x)\chi_0(\eta)\psi_j(z)\\
= &
-\frac{1}{2b_0^{1/2}}\left(\phi_{k+1}(x)+2k\phi_{k-1}(x)\right)\chi_0(\eta)(2c_{33}\eta
{z}^2+\beta_3(D_zz+zD_z))\psi_j(z)\\ &
-\frac{1}{4b_0}(\beta_3^2+2b_0c_{33})\left(\phi_{k+2}(x)
+4k(k-1)\phi_{k-2}(x)\right)\chi_0(\eta)z^2\psi_j(z)\,.
\end{aligned}
\end{multline*}
Therefore, using the identity
\[
(\mathfrak h_1-\lambda_0)\phi_{k+\ell}=2\ell b_0\phi_{k+\ell},
\]
we get
\begin{equation}\label{e:u2}
u_2(x,\eta,z)=U_2(x,\eta,z)+\phi_k(x)v_2(\eta,z)\,,
\end{equation}
where
\begin{multline*}
U_2(x,\eta,z)\\
\begin{aligned}
= &
-\frac{1}{4b_0^{3/2}}\left(\phi_{k+1}(x)-2k\phi_{k-1}(x)\right)\chi_0(\eta)(2c_{33}\eta
{z}^2+\beta_3(D_zz+zD_z))\psi_j(z)\\ &
-\frac{1}{16b^2_0}(\beta_3^2+2b_0c_{33})\left(\phi_{k+2}(x)
-4k(k-1)\phi_{k-2}(x)\right)\chi_0(\eta)z^2\psi_j(z)\,,
\end{aligned}
\end{multline*}
and $v_2\in \mathcal S(\mathbb R^2)$ has to be
determined in the next steps.

\subsection{Fourth equation: coefficient of $h^{\frac 34}$}~\\
In order to solve the equation \eqref{eq4}, we find as a necessary
condition that, for any $(z,\eta)$  we should have
\begin{equation}\label{eq:step4}
- \frac{1}{\|H_k\|^2} \langle P_3 u_0, \phi_k\rangle_x -
\frac{1}{\|H_k\|^2} \langle P_2 u_1, \phi_k\rangle_x + \lambda_2
v_1(z,\eta) +\lambda_3 \chi_0(\eta)\psi_j(z)=0.
\end{equation}

For the first term in the left-hand side of \eqref{eq:step4}, since
$\phi^2_k$ is even, we have
\[
\frac{1}{\|H_k\|^2}\langle P_3 u_0,
\phi_k\rangle_x=\frac{1}{\|H_k\|^2}\langle P^+_3 u_0,
\phi_k\rangle_x,
\]
where $P^+_3$ is the even part of $P_3$ as a differential operator
in $x$. As a sum of homogeneous components in $(\eta, D_\eta)$, the
operator $P^+_3$ is written as
\[
P^+_3= P^{(2)}_3+P^{(1)}_3+P^{(0)}_3\,,
\]
where
\begin{align*}
P^{(2)}_3 = & \left(\beta_1\frac{\eta^2}{b^2_0}
-2\beta_2\frac{\eta}{b_0}D_\eta
-\alpha_2D_\eta^2\right)D_{z}, \\
P^{(1)}_3 = & \Big[2q_{13}x^2\frac{\eta}{b_0}
+q_{23}\left(\frac{\eta}{b^2_0}(xD_x+D_xx)-2x^2D_\eta\right)
+\alpha_2\beta_3\frac{1}{b_0}(xD_x+D_xx)D_\eta\Big]{z}\\
& +\Big[\frac{\beta_2\beta_3}{2b^2_0}\eta
+\frac{\alpha_2\beta_3}{2b_0}D_\eta
+b_{33}\frac{\eta}{b_0}+a_{33}D_\eta\Big](z^2D_{z}+D_zz^2),\\
 P^{(0)}_3 =& \frac{q_{23}\beta_3}{b_0}
x^2z^3 +2(\beta_3b_{33}+b_0r_{333})x^2{z}^3 +(\beta_1x^2
+2\beta_2\frac{1}{b_0}xD_x
-\alpha_2\frac{1}{b^2_0}D^2_x)D_{{z}}\\
&
-\Big[\frac{\alpha_2\beta^2_3}{8b^2_0}+\frac{\beta_3a_{33}}{2b_0}\Big]
(D_{{z}}{z}^4
+z^4D_{{z}})-\Big[\frac{\alpha_2\beta^2_3}{2b^2_0}+\frac{\beta_3a_{33}}{b_0}\Big](xD_x+D_xx)
z^3\,.
\end{align*}

Now we compute:
\begin{multline}\label{eq:P3+}
\frac{1}{\|H_k\|^2}\langle P^+_3u_0,\phi_k\rangle_x \\ =
\frac{1}{\|H_k\|^2}\langle P^{(2)}_3u_0,\phi_k\rangle_x+
\frac{1}{\|H_k\|^2}\langle P^{(1)}_3u_0,\phi_k\rangle_x+
\frac{1}{\|H_k\|^2}\langle P^{(0)}_3u_0,\phi_k\rangle_x\,,
\end{multline}
where
\begin{multline*}
\frac{1}{\|H_k\|^2}\langle P^{(2)}_3u_0,\phi_k\rangle_x\\
\begin{aligned}
 = &
\left(\beta_1\frac{\eta^2}{b^2_0} -2\beta_2\frac{\eta}{b_0}D_\eta
-\alpha_2D_\eta^2\right)\chi_0(\eta)D_{z}\psi_j(z)\\
= & \left(\beta_1\frac{\eta^2}{b^2_0}
-2\beta_2\frac{\eta}{b_0}D_\eta
-\alpha_2D_\eta^2\right)\chi_0(\eta)\frac12\Lambda_2^{1/2}i
\left(\psi_{j+1}(z)-2j\psi_{j-1}(z)\right)\,,
\end{aligned}
\end{multline*}

\begin{multline*}
\frac{1}{\|H_k\|^2}\langle P^{(1)}_3u_0,\phi_k\rangle_x\\
\begin{aligned}
 = &
(2k+1)\Big[q_{13}\frac{\eta}{b^2_0}-q_{23}\frac{1}{b_0}D_\eta\Big]\chi_0(\eta)z\psi_j(z)\\
& +\Big[\frac{\beta_2\beta_3}{2b^2_0}\eta
+\frac{\alpha_2\beta_3}{2b_0}D_\eta
+b_{33}\frac{\eta}{b_0}+a_{33}D_\eta\Big]\chi_0(\eta)(z^2D_{z}+D_zz^2)\psi_j(z)\\
= &
(2k+1)\Big[q_{13}\frac{\eta}{b^2_0}-q_{23}\frac{1}{b_0}D_\eta\Big]\chi_0(\eta)
\frac{1}{2\Lambda_2^{1/2}}\left(\psi_{j+1}(z)+2j\psi_{j-1}(z)\right)\\
& +\Big[\frac{\beta_2\beta_3}{2b^2_0}\eta
+\frac{\alpha_2\beta_3}{2b_0}D_\eta
+b_{33}\frac{\eta}{b_0}+a_{33}D_\eta\Big]\chi_0(\eta)\frac{1}{4\Lambda_2^{1/2}}i[\psi_{j+3}(z)+\\
& +(2j+2)\psi_{j+1}(z)
-4j^2\psi_{j-1}(z)-8j(j-1)(j-2)\psi_{j-3}(z)]\,,
\end{aligned}
\end{multline*}

\begin{multline*}
\frac{1}{\|H_k\|^2}\langle P^{(0)}_3u_0,\phi_k\rangle_x\\
\begin{aligned}
 = &
(2k+1)\Big[\frac{q_{23}\beta_3}{2b^2_0}
+\frac{\beta_3b_{33}+b_0r_{333}}{b_0}\Big]\chi_0(\eta){z}^3\psi_j(z)
\\ & +(\beta_1\frac{2k+1}{2b_0}
+\beta_2\frac{1}{b_0}i-\alpha_2\frac{2k+1}{2b_0})\chi_0(\eta)D_{z}\psi_j(z)\\
& - \Big[\frac{\alpha_2\beta^2_3}{8b^2_0}
+\frac{\beta_3a_{33}}{2b_0}\Big]\chi_0(\eta)(z^4D_z+D_zz^4)\psi_j(z)\\
= & (2k+1)\Big[\frac{q_{23}\beta_3}{2b^2_0}
+\frac{\beta_3b_{33}+b_0r_{333}}{b_0}\Big]\chi_0(\eta)
\frac{1}{8\Lambda_2^{3/2}}(\psi_{j+3}(z)+
\\ & +(6j+6)\psi_{j+1}(z)+12j^2\psi_{j-1}(z)+8j(j-1)(j-2)\psi_{j-3}(z))
\\ & +(\beta_1\frac{2k+1}{2b_0}
+\beta_2\frac{1}{b_0}i-\alpha_2\frac{2k+1}{2b_0})\chi_0(\eta)\frac12\Lambda_2^{1/2}i(\psi_{j+1}(z)
-2j\psi_{j-1}(z))\\
& - \Big[\frac{\alpha_2\beta^2_3}{8b^2_0}
+\frac{\beta_3a_{33}}{2b_0}\Big]\chi_0(\eta)\frac{1}{16\Lambda^{3/2}_2}i\Big(
\psi_{j+5}(z)+(6j+12)\psi_{j+3}(z)\\
& +4(2j^2+4j+3)\psi_{j+1}(z) -8(2j^2+1)j\psi_{j-1}(z)\\
& -48j(j-1)^2(j-2)\psi_{j-3}(z)
 -32j(j-1)(j-2)(j-3)(j-4)\psi_{j-5}(z)\Big)\,.
\end{aligned}
\end{multline*}

For the second term in the left-hand side of \eqref{eq:step4}, as in
\eqref{eq3b}, we obtain
\[
\frac{1}{\|H_k\|^2}\langle P_2u_1,\phi_k\rangle_x= \mathfrak
h_3v_1(\eta,z)\,.
\]

Thus, the condition \eqref{eq:step4} can be written as an equation
with respect to $v_1$ in the form
\begin{equation}\label{eq:step4a}
(\mathfrak h_3-\lambda_2)v_1(\eta,z)=-\frac{1}{\|H_k\|^2}\langle
P^+_3u_0,\phi_k\rangle_x +\lambda_3\chi_0(\eta)\psi_j(z)\,.
\end{equation}

In order to solve \eqref{eq:step4a}, as a necessary condition, we
should have
\[
-\frac{1}{\|H_k\|^2\|H_j\|^2}\langle
P^+_3u_0,\phi_k\psi_j\rangle_{x,z} +\lambda_3\chi_0(\eta)=0\,,
\]
with $\langle\cdot,\cdot\rangle_{x,z}$ denoting the scalar product
in $L^2({\mathbb R}^2_{x,z})$. It follows from \eqref{eq:P3+} that
the function $\langle P^+_3u_0,\phi_k\rangle_x(\eta,z)\psi_j(z)$ is
odd in $z$, therefore, $\langle
P^+_3u_0,\phi_k\psi_j\rangle_{x,z}=0$ and
\begin{equation}\label{e:lambda3}
\lambda_3=0\,.
\end{equation}
Using \eqref{eq:P3+} and the identity
\[
(\mathfrak h_3-\lambda_2)\psi_{j+\ell}=2\ell\Lambda_2\psi_{j+\ell},
\]
we find a solution to \eqref{eq:step4a} in the form
\begin{equation}\label{eq:v1}
v_1= v^{(2)}_1+v^{(1)}_1+v^{(0)}_1\,,
\end{equation}
where
\begin{align*}
v^{(2)}_1 = & -\left(\beta_1\frac{\eta^2}{b^2_0}
-2\beta_2\frac{\eta}{b_0}D_\eta
-\alpha_2D_\eta^2\right)\chi_0(\eta)\frac{1}{4\Lambda_2^{1/2}}i
\left(\psi_{j+1}(z)+2j\psi_{j-1}(z)\right)\,,
\\
v^{(1)}_1 = &
-\Big[q_{13}\frac{\eta}{b^2_0}-q_{23}\frac{1}{b_0}D_\eta\Big]\chi_0(\eta)
\frac{2k+1}{4\Lambda_2^{3/2}}\left(\psi_{j+1}(z)-2j\psi_{j-1}(z)\right)\\
& -\Big[\frac{\beta_2\beta_3}{2b^2_0}\eta
+\frac{\alpha_2\beta_3}{2b_0}D_\eta
+b_{33}\frac{\eta}{b_0}+a_{33}D_\eta\Big]\chi_0(\eta)
\frac{1}{24\Lambda_2^{3/2}}i[\psi_{j+3}(z)+\\ & +(6j+6)\psi_{j+1}(z)
+12j^2\psi_{j-1}(z)+8j(j-1)(j-2)\psi_{j-3}(z)]\,,\\
v^{(0)}_1 = & -\Big[\frac{q_{23}\beta_3}{2b^2_0}
+\frac{\beta_3b_{33}+b_0r_{333}}{b_0}\Big]\chi_0(\eta)
\frac{2k+1}{48\Lambda_2^{5/2}}(\psi_{j+3}(z)+
\\ & +(18j+18)\psi_{j+1}(z)-36j^2\psi_{j-1}(z)-8j(j-1)(j-2)\psi_{j-3}(z))
\\ & -(\beta_1\frac{2k+1}{2b_0}
+\beta_2\frac{1}{b_0}i-\alpha_2\frac{2k+1}{2b_0})\chi_0(\eta)
\frac{1}{4\Lambda_2^{1/2}}i(\psi_{j+1}(z)+2j\psi_{j-1}(z))\\
& + \Big[\frac{\alpha_2\beta^2_3}{8b^2_0}
+\frac{\beta_3a_{33}}{2b_0}\Big]\chi_0(\eta)\frac{1}{16\Lambda^{5/2}_2}i
\Big( \frac{1}{10}\psi_{j+5}(z)+(j+2)\psi_{j+3}(z)+\\ &
+2(2j^2+4j+3)\psi_{j+1}(z)+4(2j^2+1)j \psi_{j-1}(z)\\
& +8j(j-1)^2(j-2)\psi_{j-3}(z)
+\frac{16j(j-1)(j-2)(j-3)(j-4)}{5}\psi_{j-5}(z)\Big)\,.
\end{align*}

With such a choice of $v_1$, the orthogonality
condition~\eqref{eq:step4} holds, and a solution $u_3$ to
\eqref{eq4} exists. Substituting \eqref{eq:v1} in \eqref{tnf14}, we
find $u_1$.

\subsection{Fifth equation: coefficient of $h$}~\\
The necessary condition for solvability of \eqref{eq5} reads
\begin{multline}\label{eq:step5}
- \frac{1}{\|H_k\|^2}\langle P_4 u_0, \phi_k\rangle_x -
\frac{1}{\|H_k\|^2}\langle P_3 u_1, \phi_k\rangle_x-
\frac{1}{\|H_k\|^2}\langle P_2 u_2, \phi_k\rangle_x \\ +\lambda_2
\langle u_2, \phi_k\rangle_x +\lambda_4 \chi_0(\eta)\psi_j(z)=0\,.
\end{multline}
Using \eqref{e:u2} and \eqref{eq3b}, we get
\[
\frac{1}{\|H_k\|^2}\langle P_2 u_2,
\phi_k\rangle_x=\frac{1}{\|H_k\|^2}\langle P_2 U_2,
\phi_k\rangle_x+\mathfrak h_3v_2(\eta,z)\,,
\]
and
\[
\frac{1}{\|H_k\|^2}\langle u_2, \phi_k\rangle_x= v_2(\eta,z)\,.
\]
Therefore, the condition \eqref{eq:step5} can be rewritten as an
equation with respect to $v_2$ in the form
\begin{multline}\label{eq:step5a}
(\mathfrak h_3-\lambda_2)v_2(\eta,z)\\ = -
\frac{1}{\|H_k\|^2}\langle P_4 u_0, \phi_k\rangle_x -
\frac{1}{\|H_k\|^2}\langle P_3 u_1, \phi_k\rangle_x-
\frac{1}{\|H_k\|^2} \langle P_2 U_2, \phi_k\rangle_x +\lambda_4
\chi_0(\eta)\psi_j(z)\,.
\end{multline}
The necessary condition for solvability of \eqref{eq:step5a} is
given by
\begin{multline}\label{eq5ort}
- \frac{1}{\|H_k\|^2\|H_j\|^2}\langle P_4 u_0,
\phi_k\psi_j\rangle_{x,z} - \frac{1}{\|H_k\|^2\|H_j\|^2}\langle P_3
u_1, \phi_k\psi_j\rangle_{x,z}\\ -
\frac{1}{\|H_k\|^2\|H_j\|^2}\langle P_2 U_2,
\phi_k\psi_j\rangle_{x,z} +\lambda_4 \chi_0(\eta)=0\,.
\end{multline}

\begin{lem}\label{l:5ort}
The identity \eqref{eq5ort} has the form
\begin{multline}\label{e:chi0}
\frac{2k+1}{2b_0}\frac{q_{22}q_{33}-q^2_{23}}{q_{33}}D^2_\eta\chi_0(\eta)-\frac{2k+1}{2b^2_0}
\frac{q_{12}q_{33}-q_{13}q_{23}}{q_{33}}(\eta
D_\eta+D_\eta\eta)\chi_0(\eta)\\+\frac{2k+1}{2b_0^3}\frac{q_{11}q_{33}-q_{13}^2}{q_{33}}\eta^2\chi_0(\eta)
 +\alpha D_\eta\chi_0(\eta)+\beta \eta\chi_0(\eta)+\gamma
\chi_0(\eta) -\lambda_4\chi_0(\eta) =0\,,
\end{multline}
where $\alpha=\alpha(j,k),\beta=\beta(j,k), \gamma=\gamma(j,k)$ are
of the form
%(see \eqref{e:alpha}, \eqref{e:beta} and \eqref{e:gamma}):
\begin{equation}\label{eq:alpha}
\alpha = \bar\alpha (2k+1)^{1/2}(2j+1),\quad \beta = \bar\beta
(2k+1)^{1/2}(2j+1),
\end{equation}
and
\begin{equation}\label{eq:gamma}
\gamma=\gamma_1(2j+1)^2+\gamma_2(2k+1)^2+\gamma_0,
\end{equation}
where $\bar\alpha$, $\bar\beta$, $\gamma_0$, $\gamma_1$ and
$\gamma_2$ are some explicit constants.
\end{lem}

The proof of Lemma~\ref{l:5ort} is given by a long routine
computation and will be omitted.
%The proof of this lemma is given in Appendix A.

By Lemma~\ref{l:5ort}, $\lambda_4$ is an eigenvalue of the second
order differential operator
\begin{multline*}
{\mathfrak h}_5:=
\frac{2k+1}{2b_0}\frac{q_{22}q_{33}-q^2_{23}}{q_{33}}D^2_\eta-\frac{2k+1}{2b^2_0}
\frac{q_{12}q_{33}-q_{13}q_{23}}{q_{33}}(\eta
D_\eta+D_\eta\eta)\\+\frac{2k+1}{2b_0^3}\frac{q_{11}q_{33}-q_{13}^2}{q_{33}}\eta^2
+\alpha D_\eta+\beta \eta+\gamma\,.
\end{multline*}
This operator is a globally elliptic operator in ${\mathbb R}$.
Therefore, it has discrete spectrum in $L^2(\mathbb R)$, which is
described by the next lemma.
\begin{lem}\label{l:harm}
The eigenvalues of the operator
\[
\mathcal H=AD^2_\eta +B(\eta D_\eta+D_\eta\eta) +C\eta^2 +\alpha
D_\eta+\beta \eta+\gamma,\quad AC-B^2>0\,,
\]
are given by
\[
\hat \lambda_m=
(2m+1)\sqrt{AC-B^2}+\gamma-\frac{1}{4(AC-B^2)}(C\alpha^2-2B\alpha\beta+A\beta^2)\,,
\quad m\in\mathbb N\,.
\]
\end{lem}

\begin{proof}
The operator $\mathcal H$ can be written as
\[
\mathcal H=XX^{+}+\sqrt{AC-B^2}+\gamma-|z|^2\,,
\]
where
\begin{align*}
X=& A^{1/2}D_\eta+\frac{B+i\sqrt{AC-B^2}}{A^{1/2}}\eta+z, \\
X^+=& A^{1/2}D_\eta+\frac{B-i\sqrt{AC-B^2}}{A^{1/2}}\eta+\bar z,\\
z =& \frac{1}{2A^{1/2}}\alpha+i\frac{A\beta-B\alpha
}{2A^{1/2}\sqrt{AC-B^2}},
\end{align*}
and we have
\[
[X^+,X]=2\sqrt{AC-B^2}\,.
\]
Therefore, the lemma is immediately proved by following a well-known
computation of the spectrum of the harmonic oscillator by means of
creation and annihilation operators.
\end{proof}

For ${\mathfrak h}_5$, we have
\[
A=\frac{2k+1}{2b_0}\frac{q_{22}q_{33}-q^2_{23}}{q_{33}}\,,
B=-\frac{2k+1}{2b^2_0} \frac{q_{12}q_{33}-q_{13}q_{23}}{q_{33}}\,,
C=\frac{2k+1}{2b_0^3}\frac{q_{11}q_{33}-q_{13}^2}{q_{33}}\,,
\]
and $\alpha,\beta$ and $\gamma $ are given by \eqref{eq:alpha} and
\eqref{eq:gamma}. In particular, by \eqref{eq:Q}, we have
\[
AC-B^2=\frac{(2k+1)^2}{4b^4_0}\frac{\det
Q}{q_{33}}=\frac{(2k+1)^2}{4b^2_0}\frac{d}{2a}.
\]

By Lemma~\ref{l:harm}, the eigenvalues of ${\mathfrak h}_5$ have the
form
\begin{equation}\label{e:lambda4}
\hat \lambda_m
=\frac{1}{2b_0}\left(\frac{d}{2a}\right)^{1/2}(2k+1)(2m+1)+
\nu(j,k)\,,\quad m\in {\mathbb N},
\end{equation}
where
\begin{align*}
\nu(j,k)=& \gamma-\frac{1}{4(AC-B^2)}(C\alpha^2-2B\alpha\beta+A\beta^2)\\
=&\nu_{22}(2k+1)^2+\nu_{11}(2j+1)^2+\nu_0\,,
\end{align*}
with
\[
\nu_{22}=\gamma_2, \quad \nu_0=\gamma_0,
\]
and
\[
\nu_{11}=\gamma_1-\frac{a}{d}
\left(\frac{1}{b_0}\frac{q_{11}q_{33}-q_{13}^2}{q_{33}}
\bar\alpha^2+2
\frac{q_{12}q_{33}-q_{13}q_{23}}{q_{33}}\bar\alpha\bar\beta
+b_0\frac{q_{22}q_{33}-q^2_{23}}{q_{33}}\bar\beta^2\right)\,.
\]

Taking $\lambda_4$ as the $m$-th eigenvalue of ${\mathfrak h}_5$ and
$\chi_0$ as the corresponding eigenfunction, we obtain that the
condition \eqref{eq5ort} holds. Therefore, we can find a solution
$v_2$ to \eqref{eq:step5a}. Then the condition \eqref{eq:step5}
holds, that allows us to find a solution $u_4$ to \eqref{eq5}.

Thus, for any $j,k,m\in \mathbb N$, we have proved the existence of
a solution $u_\ell, \ell=0,1,2,3,4,$ to the system of equations
\eqref{eq1}--\eqref{eq5} with a suitable choice of $\lambda_\ell,
\ell=0,1,2,3,4$ given by \eqref{tnf12}, \eqref{tnf15},
\eqref{tnf18}, \eqref{e:lambda3} and \eqref{e:lambda4}. Setting
\[
u^h(x,\eta,z)= \sum_{\ell=0}^4 h^{\frac \ell 4} u_\ell(x,\eta,z)\,,
\quad \lambda^h= \sum_{\ell=0}^4 h^{\frac \ell 4} \lambda_\ell,
\]
we obtain
\[
P^hu^h-\lambda^hu^h=\mathcal O(h^\frac 54).
\]
The functions $u_\ell$ have sufficient decay properties. Therefore,
by changing back to the original coordinates and multiplying by a
fixed cut-off function, we obtain the desired functions
$\phi^h_{j,k,m}$, which satisfy \eqref{e:Hh} with
\[
\mu^h_{j,k,m}= \sum_{\ell=0}^4 h^{\frac \ell 4} \mu_{j,k,m,\ell},
\]
where
\[
 \mu_{j,k,m,\ell}=\lambda_\ell, \quad \ell =0,1,2,3,4.
 \]

The functions $u_0$ for different $j$, $k$ and $m$ are orthogonal.
Since each change of variables, which we use, is unitary, this
implies the condition~\eqref{e:orth}.

\section{Periodic case and spectral gaps}\label{s:gaps}
In this section, we apply the previous results to the problem of
existence of gaps in the spectrum of a periodic magnetic
Schr\"odinger operator. For related results on spectral gaps for
periodic magnetic Schr\"odinger operators, see \cite{luminy} and
references therein.

Consider the Schr\"odinger operator with magnetic potential in the
entire Euclidean space ${\mathbb R}^3$
\[
H^h=\left(hD_{X_1}-A_1(X)\right)^2+\left(hD_{X_2}-A_2(X)\right)^2+\left(hD_{X_3}-A_3(X)\right)^2.
\]
We assume that the vector magnetic field $\vec{B}=(B_1,B_2,B_3)$ is
$\Gamma$-invariant for some cocompact lattice $\Gamma \subset
{\mathbb R}^3$.

Put
\[
b_0=\min \{|\vec B(X)|\, :\, X\in {\mathbb R}^3\}
\]
and assume that there exist a (connected) fundamental domain $Q$ for
$\Gamma$ and a constant $\epsilon_0>0$ such that
\[
 |B(X)| \geq b_0+\epsilon_0, \quad x\in \partial Q\,.
\]

We will consider the operator $H^h$ as an unbounded self-adjoint
operator in the Hilbert space $L^2({\mathbb R}^3)$. Using the
results of \cite{diff2006}, one can immediately derive from
Theorem~\ref{t:qmodes} the following result on existence of gaps in
the spectrum of $H^h$ in the semiclassical limit.

\begin{thm}~\\
Suppose that
\[
b_0>0\,,
\]
 and that there exists a unique minimum $X_0\in Q$ such that
  $|B(X_0)|=b_0$, which is non-degenerate: in some neighborhood of $X_0$
\[
C^{-1}|X-X_0|^2\leq |\vec B(X)|-b_0 \leq C |X-X_0|^2\,.
\]
Then, for any natural $j$, $k$ and $N$, there exist $h_{j,k,N}>0$,
$c_{j,k}$ and $C_{j,k,N}$ such that
\[
C_{j,k,N}>c_{j,k}+\frac{1}{b_0}\left(\frac{d}{2a}\right)^{1/2}(2k+1)N,
\]
and the spectrum of $H^h$ in the interval $[A_{j,k},B_{j,k,N}]$
where
\[
A_{j,k}=(2k+1)b_0 + (2j+1)(2k+1)^{1/2}a^{1/2}h^{3/2}+ h^2c_{j,k}
\]
and
\[
B_{j,k,N}= (2k+1)b_0 + (2j+1)(2k+1)^{1/2}a^{1/2}h^{3/2} +
h^2C_{j,k,N}\,,
\]
has at least $N$ gaps for any $h\in ]0, h_{j,k,N}]\,$.
\end{thm}

\appendix
\section{Hermite polynomials}
For a fixed $\lambda>0$, put
\[
h_{m}(t)=\lambda^{1/2} H_{m}(\lambda^{1/2}t) e^{-\lambda t^2/2}\,.
\]
Here we gather some well-known formulae, concerning to the Hermite
functions $h_m$, which we need in the paper:
\begin{align*}
th_m(t)= & \frac{1}{2\lambda^{1/2}}[h_{m+1}(t)+2mh_{m-1}(t)]\,.\\
t^2h_m(t)= & \frac{1}{4\lambda}[h_{m+2}(t)+(4m+2)h_m(t)+4m(m-1)h_{m-2}(t)]\,.\\\
t^3h_m(t)= &
\frac{1}{8\lambda^{3/2}}[h_{m+3}(t)+(6m+6)h_{m+1}(t)+12m^2h_{m-1}(t)\\ & +8m(m-1)(m-2)h_{m-3}(t)]\,.\\
D_{t}h_{m}(t)=& \frac12\lambda^{1/2}i[h_{m+1}(t)-2mh_{m-1}(t)]\,.
\end{align*}

\[
\langle (tD_t+D_tt)h_{m},h_m\rangle =0\,.
\]

\begin{align*}
(D_{t}{t}^2+ t^2D_{t})h_m(t) =\frac{1}{4\lambda^{1/2}}i &
[h_{m+3}(t)+(2m+2)h_{m+1}(t) -4m^2h_{m-1}(t)\\ &
-8m(m-1)(m-2)h_{m-3}(t)]\,.
\end{align*}

\begin{multline*}
(t^4D_{t}+D_tt^4)h_{m}(t)\\
\begin{aligned}
 = \frac{1}{16\lambda^{3/2}}i & [
h_{m+5}(t)+(6m+12)h_{m+3}(t)+4(2m^2+4m+3)h_{m+1}(t)\\ &
-8(2m^2+1)mh_{m-1}(t) -48m(m-1)^2(m-2)h_{m-3}(t)\\ &
-32m(m-1)(m-2)(m-3)(m-4)h_{m-5}(t)]\,.
\end{aligned}
\end{multline*}

\end{document}